\documentclass{amsart}
\usepackage{tikz}
\usetikzlibrary{matrix}
\usepackage{amssymb}
\usepackage{latexsym}
\usepackage{enumerate}

\newtheorem{thm}{Theorem}[section]
\newtheorem{defin}[thm]{Definition}
\newtheorem{cor}[thm]{Corollary}
\newtheorem{prop}[thm]{Proposition}
\newtheorem{lemma}[thm]{Lemma}
\newtheorem{rmk}[thm]{Remark}
\newtheorem{fact}[thm]{Fact}

\newcommand{\integers}{\ensuremath{\mathbb{Z}}}
\newcommand{\complex}{\ensuremath{\mathbb{C}}}

\newcommand{\abs}[1]{\ensuremath{\left|#1\right|}}
\newcommand{\defsetshort}[1]{\ensuremath{\left\{#1\right\}}}
\newcommand{\defset}[2]{\defsetshort{#1\,\left|\,#2\right.}}
\newcommand{\defsetspan}[1]{\ensuremath{<\!\!#1\!\!>}}

\newcommand{\defby}{\mathrel{\mathop:}=}

\newcommand{\symgroup}[1]{\ensuremath{\mathfrak{S}_{#1}}}
\newcommand{\altgroup}[1]{\ensuremath{\mathfrak{A}_{#1}}}
\newcommand{\cyclgroup}[1]{\ensuremath{{\boldsymbol{\mu}}_{#1}}}
\newcommand{\glgroup}[2]{\ensuremath{\mathrm{GL}_{#1}\left(#2\right)}}
\newcommand{\slgroup}[3][]{\ensuremath{\mathrm{SL}^{#1}_{#2}\left(#3\right)}}
\newcommand{\pglgroup}[2]{\ensuremath{\mathrm{PGL}_{#1}\left(#2\right)}}
\newcommand{\proj}[2][]{\ensuremath{\mathbb{P}_{#1}^{#2}}}

\newcommand{\deffunname}[3]{\ensuremath{#1:#2\rightarrow#3}}
\newcommand{\defmapname}[3]{\ensuremath{#1:#2\mapsto#3}}
\title{$G$-birational rigidity of the projective plane}
\author[D.\,Sakovics]{Dmitrijs Sakovics}
\address{\emph{Dmitrijs Sakovics}, Center for Geometry and Physics, Institute for Basic Science (IBS), Pohang, 37673, Korea.}
\email{dmitry85@ibs.re.kr}
\thanks{This work has been supported by IBS-R003-D1, Institute for Basic Science in Korea.}

\begin{document}

\begin{abstract}
 Given a surface $S$ and a finite group $G$ of automorphisms of $S$, consider the birational maps $S\dashrightarrow S'$ that commute with the action of $G$. This leads to the notion of a $G$-minimal variety. A natural question arises: for a fixed group $G$, is there a birational $G$-map between two different $G$-minimal surfaces? If no such map exists, the surface is said to be $G$-birationally rigid. This paper determines the $G$-rigidity of the projective plane for every finite subgroup $G\subset\pglgroup{3}{\complex}$.
\end{abstract}
\maketitle

 \section{Introduction}

 Pick a finite group $G$. Let $S$ be a nonsingular projective surface, and take $\rho$ to be a monomorphism from $G$ to the group of automorphisms of $S$. Such a pair $\left(S,\rho\right)$ is called a \emph{$G$-surface}. Given a second $G$-surface $\left(S',\rho'\right)$, one can define a \emph{morphism of $G$-surfaces} to be a morphism \deffunname{f}{S}{S'}, such that $\rho'\left(G\right)=f\circ\rho\left(G\right)\circ f^{-1}$. Usually, the monomorphism $\rho$ is implied by the choice of $S$, and the pair $\left(S,\rho\right)$ will be referred to as simply a $G$-surface $S$ (unless there is some ambiguity). In particular, this paper will mostly deal with the case $S=\proj{2}$ and $G\subset\pglgroup{3}{\complex}=\mbox{Aut}\left(\proj{2}\right)$, making $\rho$ the inclusion map.
 
 A $G$-surface $S$ is said to be \emph{minimal} if any birational $G$-morphism of $S\rightarrow S'$ is an isomorphism of $G$-surfaces. Minimal rational $G$-surfaces have been classified as follows: 
 \begin{thm}[{\cite[Theorem~3.8]{Dolgachev-Iskovskikh09}}]
  Let $S$ be a minimal rational $G$-surface. Then either $S$ admits a structure of a conic bundle with $\mbox{Pic}\left(S\right)^{G}\cong\integers^2$, or $S$ is isomorphic to a Del~Pezzo surface with $\mbox{Pic}\left(S\right)^{G}\cong\integers$.
 \end{thm}
 
 Clearly, it is possible for the same group $G$ to have several different minimal $G$-surfaces. This leads to the $G$-equivariant version of birational rigidity:
 \begin{defin}[\cite{Dolgachev-Iskovskikh09}]
  A $G$-surface $S$ is called \emph{$G$-rigid} if for any birational $G$-map $\Psi:S\dashrightarrow S'$ there exists a birational $G$-automorphism $\alpha:S\dashrightarrow S$, such that $\Psi\circ\alpha$ is a $G$-isomorphism.
 \end{defin}
 
 In the case of $S=\proj{2}$, the group $G$ is a finite subgroup of \pglgroup{3}{\complex}. The list of finite subgroups of \pglgroup{3}{\complex} is well-known, appearing in H.F.~Blischfeldt's 1917 book~\cite{Blichfeldt17}. These groups can be divided into two types: transitive and intransitive. Intransitive groups fix a point on \proj{2}. If a group does not fix any point on \proj{2}, it is said to be transitive. Among transitive groups one can further distinguish the two classes of primitive and imprimitive groups: imprimitive groups have an orbit of size $3$, while primitive ones do not.

 This paper is dedicated to proving the following result:
 \begin{thm}\label{thm:main}
  The projective plane \proj{2} is $G$-rigid if and only if the action of $G$ is transitive and $G$ is not isomorphic to \altgroup{4} or \symgroup{4}.
 \end{thm}
  
 The remainder of this article is structured as follows: Section~$2$ will talk about the possible birational $G$-maps that may serve as a counterexample to $G$-rigidity. This will imply that for a \proj{2} not to be $G$-rigid, the action of $G$ must have an orbit of size at most $8$, and Section~$3$ will find the list of groups that have such orbits. After that, Section~$4$ will consider each of these groups and determine if a counterexample to $G$-rigidity can be constructed.

 \section{Elementary links between $G$-surfaces}
 
 To study the $G$-rigidity of \proj{2}, first consider the four special classes of $G$-maps of $G$-surfaces $\chi:S\dashrightarrow S'$ called elementary links (or $G$-links). The $G$-map $\chi$ is called an elementary link of Type~I, II or~IV if it satisfies the respective diagram in Figure~\ref{figure:links}. In all the diagrams, $\pi$ and~$\pi'$ are blowups of $G$-orbits, while $\phi$ and~$\phi'$ are $G$-fibrations. The surface $Z$ is a Del~Pezzo surface, and an elementary link of Type~IV is an exchange of two $G$-equivariant conic bundle structures on $S$. The map $\chi$ is called an elementary link of Type~III if it is the inverse of an elementary link of Type~I.
 
 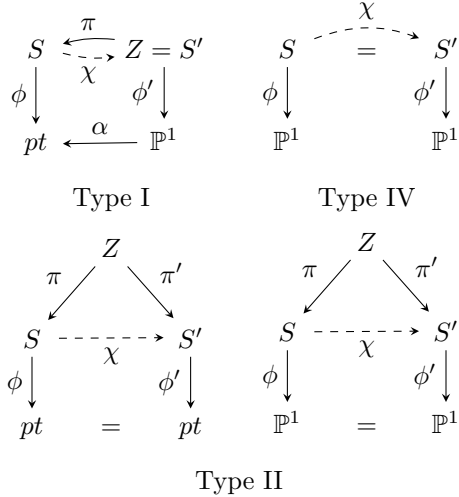
\begin{figure}[htbp]
 \begin{center}
  \begin{tabular}{cc}
  \begin{tikzpicture}
   \matrix (m) [matrix of math nodes,row sep=2em,column sep=1em,minimum width=2em]
   {
     S&&Z=S'\\
     pt&&\proj{1}\\
   };
   \path[-stealth]
    (m-1-3) edge[bend right = 10] node [above] {$\pi$} (m-1-1)
    (m-1-1) edge[dashed, bend right = 10] node [below] {$\chi$} (m-1-3)
    (m-1-1) edge node [left] {$\phi$} (m-2-1)
    (m-1-3) edge node [left] {$\phi'$} (m-2-3)
    (m-2-3) edge node [above] {$\alpha$} (m-2-1)
   ;
  \end{tikzpicture}&
  \begin{tikzpicture}
   \matrix (m) [matrix of math nodes,row sep=2em,column sep=1em,minimum width=2em]
   {
     S&=&S'\\
     \proj{1}&&\proj{1}\\
   };
   \path[-stealth]
    (m-1-1) edge[dashed, bend left = 20] node [above] {$\chi$} (m-1-3)
    (m-1-1) edge node [left] {$\phi$} (m-2-1)
    (m-1-3) edge node [left] {$\phi'$} (m-2-3)
   ;
  \end{tikzpicture}\\
  Type~I&Type~IV\\
  \begin{tikzpicture}
   \matrix (m) [matrix of math nodes,row sep=2em,column sep=1em,minimum width=2em]
   {
     &Z&\\
     S&&S'\\
     pt&=&pt\\
   };
   \path[-stealth]
    (m-1-2) edge node [above left] {$\pi$} (m-2-1)
    (m-1-2) edge node [above right] {$\pi'$} (m-2-3)
    (m-2-1) edge[dashed] node [below] {$\chi$} (m-2-3)
    (m-2-1) edge node  [left] {$\phi$} (m-3-1)
    (m-2-3) edge node [left] {$\phi'$} (m-3-3)
   ;
  \end{tikzpicture}&
  \begin{tikzpicture}
   \matrix (m) [matrix of math nodes,row sep=2em,column sep=1em,minimum width=2em]
   {
     &Z&\\
     S&&S'\\
     \proj{1}&=&\proj{1}\\
   };
   \path[-stealth]
    (m-1-2) edge node [above left] {$\pi$} (m-2-1)
    (m-1-2) edge node [above right] {$\pi'$} (m-2-3)
    (m-2-1) edge[dashed] node [below] {$\chi$} (m-2-3)
    (m-2-1) edge node [left] {$\phi$} (m-3-1)
    (m-2-3) edge node [left] {$\phi'$} (m-3-3)
   ;
  \end{tikzpicture}\\
  \multicolumn{2}{c}{Type~II}\\
  \end{tabular}
  \caption{Elementary links.}\label{figure:links}
 \end{center}
 \end{figure}
 
 Elementary links are important for the study of $G$-rigidity because of the following result:
 \begin{thm}[\cite{Corti95}]
  Let $f:S\dashrightarrow S'$ be a birational $G$-map of minimal $G$-surfaces. Then $f$ is equal to a composition of elementary links.
 \end{thm}

 This result means that in order to study the $G$-rigidity of \proj{2} (for different groups $G$) one needs to consider all the elementary links (of Types~I and~II) with $S=\proj{2}$. To do this, consider the blowup $\pi$. Since $Z$ is a Del~Pezzo surface, $\pi$ must be a blowup of at most $8$ points in general position. On the other hand, since $\pi$ is a $G$-map, it has to blow up complete $G$-orbits on $S$. Therefore, one needs to determine which groups $G\subset\pglgroup{3}{\complex}$ have orbits of size at most $8$ with the orbit's points being in general position. Recall that a Del~Pezzo surface is a blowup of \proj{2} in up to $8$ points in general position, where:
 \begin{defin}\label{defin:genPos}
  Points $p_1,\ldots,p_n\in\proj{2}$ are said to be \emph{in general position} if
  \begin{itemize}
   \item No three of these points lie on the same line
   \item No six of these points lie on the same conic curve.
   \item No eight of these points lie on a plane cubic which has one of the points as its singular point.
  \end{itemize}
 \end{defin}
 
 The results above immediately imply the following:
 \begin{cor}
  Let $G$ be a group acting on \proj{2}. If all $G$-orbits on \proj{2} have size greater than $8$, then \proj{2} is $G$-birationally rigid and $\mbox{Bir}^G\left(\proj{2}\right)=\mbox{Aut}^G\left(\proj{2}\right)$.
 \end{cor}

 \section{Groups acting on \proj{2} with small orbits}
 Let $G\subset\glgroup{3}{\complex}$ be a finite group. Then let $\bar{G}\subset\pglgroup{3}{\complex}$ be the image of $G$ under the natural projection from $\glgroup{3}{\complex}$ onto $\pglgroup{3}{\complex}$. Since this paper studies the action of $\bar{G}$ (rather than $G$), it can be assumed that $G\subset\slgroup{3}{\complex}$. Furthermore, whenever any subgroup $H\subset\slgroup{3}{\complex}$ is used, take $\bar{H}$ to mean its image under this projection to \pglgroup{3}{\complex}.
 
 The action of $G$ not being irreducible is equivalent to the action of $\bar{G}$ being intransitive (i.e.\ having a fixed point on \proj{2}). This case is easy to resolve (see Proposition~\ref{prop:link:orb1}), so for now assume that the action of $G$ is irreducible. In this case, the action is either monomial (i.e.\ the group $\bar{G}$ is transitive but imprimitive) or primitive (i.e.\ the group $\bar{G}$ is primitive). These two cases need to be considered separately.
 
 \subsection{Monomial group actions}
 \subsubsection{Group structure and notation}\label{sect:grpStruct}
  Assume the group $G$ acts monomially. The aim of this section is to classify such groups $G$ (or $\bar{G}$), where $\bar{G}$ has no orbits of size $1$ or $2$, but does have an orbit of size $4$, $5$, $6$, $7$ or $8$. Let $X=(x:y:z)$ be a member of such an orbit. Note that since $G$ is monomial, $\bar{G}$ does have an orbit of size $3$. Call this a \emph{distinguished} orbit, and choose a basis for $\complex^3$ (and, hence, \proj{2}) in which this orbit consists of points $(1:0:0)$, $(0:1:0)$ and $(0:0:1)$.
 Without loss of generality, can assume that $G\subset\slgroup{3}{\complex}$.
 
 Since $G$ is monomial, the action of $G$ permutes the chosen basis, and hence have an exact sequence:
 \[
  0\longrightarrow D\longrightarrow G\longrightarrow\symgroup{3},
 \]
 where $D$ is a subgroup of diagonal matrices (in the chosen basis). The image of $G$ in \symgroup{3} is either \cyclgroup{3} or $\symgroup{3}$. Let $\tau\in G$ be an element, whose image is a $3$-cycle in \symgroup{3}, and let $\sigma\in G$ be an element (if such exists), whose image is a $2$-cycle in $\symgroup{3}$. Let $T$ be the group generated by $D$ and $\tau$. It is clear that either $G=T$ or $G$ is generated by $T$ and $\sigma$.
 
 \begin{prop}
  It is possible to choose a basis as above, in which $\tau$ acts as the matrix \[M_{\tau}=\left(\begin{array}{ccc}0&1&0\\0&0&1\\1&0&0\end{array}\right).\]
 \end{prop}
 \begin{proof} Take the basis for $\complex^3$, whose elements project to the distinguished orbit in \proj{2}. Note that the matrix representing $\tau$ needs to have determinant $1$.\end{proof}
 
 From now on, fix this choice of a basis for $\complex^3$.
 Since $D$ is a finite group of diagonal matrices, there exists a number $k\in\integers_{>0}$, such that the diagonal entries of any element of $D$ are (not necessarily primitive) $k$-th roots of unity. Pick the smallest possible such $k$, and fix $\zeta_k$ as a primitive $k$-th root of unity. Since the group $\bar{G}$ does not have an orbit of size $1$, have $k>1$. Denote by $M_k\left(a,b,c\right)$ the diagonal matrix with diagonal entries $\zeta_k^a$, $\zeta_k^b$ and $\zeta_k^c$. Clearly, any element of $D$ can be described via this notation. Furthermore, since $D\subset G\subset\slgroup{3}{\complex}$, any element of $D$ is actually of the form $M_k\left(a,b,-(a+b)\right)$.
 
 In order to classify groups $G$ that give rise to small $\bar{G}$-orbits, we will look at the projective $D$-orbits and then at their extensions to orbits of $T$ and $G$. Since $D\lhd T\lhd G$, the action of $T$ permutes the $D$-orbits and the action of $G$ permutes the $T$-orbits.
 
 \subsubsection{Classification of small $D$-orbits}
 Let $D_1\subseteq D$ be the group generated by all the elements of $D$ of the form $M_k\left(a,b,c\right)$ with $abc=0$. Clearly, $D_1$ is generated by elements $M_k\left(b,0,-b\right)$ and $M_k\left(0,b,-b\right)$, for some $b\in\integers_{>0}$. Without loss of generality, choose $b$ to be minimal such. Clearly, $b|k$. If $D_1=D$, then $b=1$. Otherwise, there exists an element $M_k\left(1,a,-(a+1)\right)\in D\setminus D_1$.
 
 \begin{prop}\label{prop:D:A-type}
  Using the notation above, if $D_1\neq D$, then the pair $(a,b)$ is one of: $(2,7)$, $(4,7)$, $(4,21)$, $(16,21)$. In these cases, $b=k$ and the group $D_1$ is trivial.
 \end{prop}
 \begin{proof}
   Assume $D_1\neq D$, and let $g_1=M_k\left(1,a,-(a+1)\right)\in D$. Since $D\lhd T$, can conjugate $g_1$ by the action of $\tau$ to obtain $g_2=M_k\left(a,-(a+1),1\right)\in D$. Take:
   \[
    g_1^ag_2^{-1}=M_k\left(a-a,a^2+a+1,-a(a+1)-1\right).
   \]
   This is an element of $D_1$, so $a^2+a+1=0\mod{b}$ (but not necessarily modulo $k$).

   Now one needs to find an upper bound for the value of $b$. To do this, consider the action of $g_1$ on $X=\left(x:y:z\right)$. Assume that $X$ is not in the distingushed orbit, i.e.\ at least two of $x,y,z$ are non-zero. Without loss of generality, assume $x,y\neq0$. Let $l$ be the smallest positive integer, such that $g_1^l\left(X\right)\in D_1\left(X\right)$. Then \[g_1^l\left(X\right)=\left(\zeta_k^lx:\zeta_k^{al}y:\zeta_k^{-l(a+a)}z\right)=\left(\zeta_k^{bt_1}x:\zeta_k^{bt_2}y:\zeta_k^{bt_3}z\right)\]
   for some $t_1,t_2,t_3$. Since $x,y\neq0$, for some $\lambda\in\complex^*$ have
   \[
    \zeta_k^l=\lambda\zeta_k^{bt_1},\ \zeta_k^{al}=\lambda\zeta_k^{bt_2}
   \]
   Thus, since $b|k$, have $al=l\mod{b}$.
   
   Combining the last two results, have $3l=0\mod{b}$. Since we are only interested in orbits of size up to $8$, have $0<l\leq8$, and so $0<b\leq24$. Since $0<a\leq b$ and $a^2+a+1=0\mod{b}$, this only leaves the four possible pairs $\left(a,b\right)$ above. In all of these cases, $l=7$. Note that if $k>b$, then the $D$-orbit of $X$ has size at least $lk/b\geq2l$. Since the orbit of $X$ needs to have size at most $8$, have $b=k$. 
 \end{proof}

 \begin{fact}
  If $D_1=D$, then $\abs{D}=k^2$. Thus, the $D$-orbit of any point has size at most $k^2$.
 \end{fact}
 
 By definition, the action of $D$ preserves the distinguihed orbit of size $3$ (forming the chosen basis). Therefore, it also preserves the set of three lines going through pairs of points of that orbit. Call these lines~$L_1$, $L_2$ and~$L_3$.
 
 \begin{prop}\label{prop:D:lines}
  If $D_1=D$, then, given a point $X$ on one of the lines $L_i$, either $X$ lies in the distinguished orbit or $X$ has a $D$-orbit of size $k$.
 \end{prop}
 \begin{proof} By direct computation. \end{proof}
 
 \begin{prop}
  Let $D_1=D$ and let $X$ be a point not contained in the lines $L_i$. If the orbit of $X$ has size less than $k^2$, then $3|k$ and the orbit of $X$ has size at least $k$.
 \end{prop}
 \begin{proof}
   The group $D$ is generated by elements $g_1=M_k\left(1,0,-1\right)$ and $g_2=M_k\left(0,1,-1\right)$. Since $X$ does not lie on the lines $L_i$, have $X=\left(x:y:z\right)$, where $xyz\neq0$. It is easy to see that $g_i^t\left(X\right)=X$ if and only if $t=0\mod{k}$ (hence the orbit of $X$ has at least $k$ elements). Therefore, need $g_1^{t_1}\left(X\right)=g_2^{t^2}\left(X\right)$ for some $t_1$ and $t_2$ not both zero modulo $k$. This gives
   \[
    \zeta_k^{t_1}x=\lambda x,\ y=\lambda\zeta_k^{t_2}y,\ \zeta_k^{-t_1}z=\lambda\zeta_k^{-t_2}z
   \]
   for some $\lambda\neq0$. Thus
   \[
    t_1=-t_2\mod{k},\ 3t_1=0\mod{k},
   \]
   and $3|k$.
 \end{proof}
 
 \begin{cor}
  In the proposition above, if we also assume that the size of the $D$-orbit is at most $8$, then $k=3$ and $\bar{D}$ does indeed have an orbit of size $3$. 
 \end{cor}
 \begin{proof}
   If $k=3$, the $\bar{D}$-orbits of generic points do indeed have size $3$. Since the orbit must have size less than $8$, it only remains to check the case $k=6$. This can be checked explicitly, showing that no suitable orbits exist.
 \end{proof}

 To summarise this section, we have the following result:
 \begin{lemma}\label{lemma:D-orb}
  Let $D$ be a group constructed at the start of this and let $X\in\proj{2}$ be a point with the projective $\bar{D}$-orbit of $X$ having size at most $8$. Then one of the following cases occur:
  \begin{enumerate}
   \item\label{D-orb:dist} $X$ is a point in the distinguished orbit, i.e.\ $X$ is one of $(1:0:0)$, $(0:1:0)$, $(0:0:1)$. The $\bar{D}$-orbit of $X$ has size $1$.
   \item\label{D-orb:lines} $X$ is not one of the three points above, but is contained in the lines $L_1,L_2,L_3$ connecting them. Then the $\bar{D}$-orbit of $X$ has size $k$ and is contained in one of the three lines.
   \item\label{D-orb:gen3} $k=3$ and $X$ has a $\bar{D}$-orbit of size $3$. Here, $X$ can be taken to be any point not contained in the three lines $L_i$ described above.
   \item\label{D-orb:gen4} $k=2$ and $X$ has a $\bar{D}$-orbit of size $4$. Here, $X$ can be taken to be any point not contained in the three lines $L_i$ described above.
   \item\label{D-orb:special} The group $D$ is one of the four groups described in Proposition~\ref{prop:D:A-type}.
  \end{enumerate}
 \end{lemma}
 \begin{proof} Immediate from the results in this section. \end{proof}
 
 \subsubsection{Building $\bar{T}$-orbits}
 In the notation of Section~\ref{sect:grpStruct}, consider the extension of the group $D$ to the group $T\subseteq G$. This is done by adding the generator $\tau$, whose action must act on the space of $\bar{D}$-orbits. To make things more clear, it is worth to consider some groups separately before assembling the final result:

 First, one should consider the groups described in Proposition~\ref{prop:D:A-type}. For clarity, let $D_{(a,k)}$ be the group generated by $g_{(a,k)}=M_k\left(1,a,-(a+1)\right)$ (where, as in Proposition~\ref{prop:D:A-type}, $\left(a,k\right)$ is one of $(2,7)$, $(4,7)$, $(4,21)$, $(16,21)$). Let $T_{(a,k)}$ be the group generated by $D_{(a,k)}$ and $\tau$. Let $W$ be the scalar matrix with all diagonal entries equal $\zeta_3$.
 \begin{prop}\label{prop:T:A-type:conjugacy}
  In the notation above, the following hold:
  \begin{itemize}
   \item $T_{(16.21)}=\defsetspan{T_{(2,7)},W}$. Hence $\bar{T}_{(2,7)}=\bar{T}_{(16,21)}$.
   \item $T_{(4.21)}=\defsetspan{T_{(4,7)},W}$. Hence $\bar{T}_{(4,7)}=\bar{T}_{(4,21)}$.
   \item $T_{(2,7)}$ and $T_{(4,7)}$ (respectively, $T_{(4,21)}$ and $T_{(16,21)}$) are conjugate in \slgroup{3}{\complex} (and hence in \pglgroup{3}{\complex}
  \end{itemize}
 \end{prop}
 \begin{proof}
  By direct computation, have 
  \[
   g_{(4,21)}^3=g_{(4,7)},\ g_{(16,21)}^3=g_{(2,7)},\mbox{ and }g_{(4.21)}^7=g_{(16,21)}^7=W
  \]
  Counting the numbers of elements in the relevant groups, get the equalities in the proposition's statement. Since the image of $W$ in \pglgroup{3}{\complex} is the identity, the pairs of groups produce the same projective actions. The conjugacies mentioned can be achieved via the element
  \[
   -\left(\begin{array}{ccc}
                 1&0&0\\
                 0&0&1\\
                 0&1&0
                \end{array}
\right)\in\slgroup{3}{\complex}.
  \]
 \end{proof}

 \begin{prop}\label{prop:T:A-type}
  Let $T_{(a,k)}$ be one of the four groups described above. Let $X$ be a point contained in a $\bar{T}$-orbit of size at most $8$. Assume further that $X$ is not one of $(1:0:0)$, $(0:1:0)$, $(0:0:1)$. Then the orbit of $X$ has size $7$ and it contains the point $\left(1:\zeta_3^c:\zeta_3^{2c}\right)$ for some $c\in\integers$. These are representatives of $3$ distinct orbits of size $7$, and the points of each of these orbits are in general position.
 \end{prop}
 \begin{proof}
  Consider the generator $g_{(a,k)}\in D_{(a,k)}$. It is easy to check directly that for each of the groups mentioned here, $D_{(a,k)}\lhd T_{(a,k)}$. Therefore, the group $\bar{T}_{(a,k)}$ has size $7\cdot3$. Since the size of the $\bar{T}_{(a,k)}$-orbit needs to be at most $8$, the action of $\tau$ needs to preserve the $\bar{D}_{(a,k)}$-orbit. To do this, need to have $\tau\left(X\right)=g_{(a,k)}^t\left(X\right)$ for some $t\in\integers$. Setting $X=\left(x:y:z\right)\in\proj{2}$, this implies
  \[
   y=\lambda\zeta_k^tx,\ z=\lambda\zeta_k^{at}y,\ x=\lambda\zeta_k^{-t(a+1)}z
  \]
  for some $\lambda\in\complex^{*}$. Therefore, $\lambda=\zeta_3^c$ (some $c\in\integers$) and $X=\left(1:\zeta_3^c\zeta_k^{t}:\zeta_3^{2c}\zeta_k^{t(a+1)}\right)$.
  
  Considering the groups $T_{(2,7)}$ and $T_{(4,7)}$, one can see that in each case such points form exactly $3$ orbits of size $7$, distinguished by the value of $c$.
  
  Since in each of the cases, the orbits are known explicitly, it is easy to check by direct computation that the seven points of any given orbit lie in general position.
  
  Since the groups $T_{(16,21)}$ and $T_{(4,21)}$ have the same projective actions as the groups $T_{(2,7)}$ and $T_{(4,7)}$ (respectively), the same conclusions hold for these two groups.
 \end{proof}

 \begin{lemma}\label{lemma:T-orb}
  Let $X\in\proj{2}$ be a point, whose $\bar{T}$-orbit has size at most $8$. Then one of the following cases occurs:
  \begin{enumerate}
   \item\label{T-orb:dist} The $\bar{T}$-orbit of $X$ consists of the three points $(1:0:0)$, $(0:1:0)$, $(0:0:1)$.
   \item\label{T-orb:6pt} $k=2$ and $X$ has a $\bar{T}$-orbit of $6$ points, two on each of the lines $L_i$ defined in Case~\ref{D-orb:lines} of Lemma~\ref{lemma:D-orb}.
   \item\label{T-orb:4pt} $k=2$, $X$ has a $\bar{T}$-orbit of size $4$ and $X$ is of the form \[X=\left(1:(-1)^a\lambda:(-1)^b\lambda^2\right)\] with $a,b\in\integers$ and $\lambda^3=1$. The three different orbits can be distinguished by the corresponding value of $\lambda$. The points of each of these orbits lie in general position.
   \item\label{T-orb:3pt} $k=3$ and $X$ has a $\bar{T}$-orbit of size $3$ and $X$ is of the form \[X=\left(1:\zeta_3^a:\zeta_3^b\right)\] with $a,b\in\integers$. There are exactly $3$ such orbits, with points of each of them being in general position.
   \item\label{T-orb:spec} The group $D$ is one of the four groups described in Proposition~\ref{prop:D:A-type} and $X$ has an orbit of size $7$. The points of this orbit are described in Proposition~\ref{prop:T:A-type}.
  \end{enumerate}
  In cases~\ref{T-orb:6pt}, \ref{T-orb:4pt}, \ref{T-orb:3pt} and~\ref{T-orb:spec}, the group $\bar{T}$ is uniquely defined up to \pglgroup{3}{\complex}-conjugation.
 \end{lemma}
 \begin{proof}
  This extends the results of Lemma~\ref{lemma:D-orb}. Since $D\lhd T$, any $\bar{T}$-orbit is either a $\bar{D}$-orbit or a union of three of them. To distinguish between the two situations, consider the cases from Lemma~\ref{lemma:D-orb} one at a time.
  \begin{itemize}
   \item In Case~\ref{D-orb:dist}, the three points are clearly permuted by the action of $\tau$, producing the distinguished orbit.
   \item The action of $\tau$ also permutes the three lines $L_1,L_2,L_3$ mentioned in Case~\ref{D-orb:lines}. Therefore, triples of $\bar{D}$-orbits of size $k$ mentioned in this case are combined into $\bar{T}$-orbits of size $3k$. Since the total size of a $\bar{T}$-orbit needs to be at most $8$, need to have $k=2$, producing a family of orbits of size $6$. Clearly, no three of the points of any given orbit lie on the same line. However, one can see by direct computation that the whole orbit is contained in a conic if and only if the orbit contains the point $\left(0:1:a\right)$ with $a^6=-1$. It is clear that there are exactly $3$ such orbits, with all other orbits having their points in general position.
   \item In Cases~\ref{D-orb:gen3} and~\ref{D-orb:gen4}, the $\bar{D}$-orbit of $X$ needs to be preserved by the action of $\tau$ (otherwise the $\bar{T}$-orbit has to have size $9$ or $12$ resp.). The required values of $X$ can be found by an easy calculation. Once the orbits are explicitly known, it is easy to check that each of them has its points in general position.
   \item The statement of Case~\ref{T-orb:spec} is discussed in full in Proposition~\ref{prop:T:A-type}.
  \end{itemize}
  Since the groups in cases~\ref{T-orb:6pt}, \ref{T-orb:4pt}, \ref{T-orb:3pt} and~\ref{T-orb:spec} are constructed as explicit matrices (in a chosen basis), they are all defined uniquely up to \pglgroup{3}{\complex}-conjugation.
 \end{proof}
 \begin{rmk}\label{rmk:A4:geometry}
  The group obtained in cases~\ref{T-orb:6pt} and~\ref{T-orb:4pt} of Lemma~\ref{lemma:T-orb} should be discussed a bit further. From the explicit description in the above results, it can be seen that this group $T$ is isomorphic to the alternating group \altgroup{4}, acting on $\complex^3$ as its irreducible $3$-dimensional representation (and projecting to \pglgroup{3}{\complex} isomorphically). For clarity, work in the basis for $\complex^3$ chosen above.
  
  As discussed in the lemma, this group has $3$ orbits of size $4$. Call them $O_1$, $O_2$ and~$O_3$ with $\left(1:1:1\right)\in O_1$, $\left(1:\zeta_3:\zeta_3^2\right)\in O_2$, $\left(1:\zeta_3^2:\zeta_3\right)\in O_3$. It is also known that this group preserves three conics, which can be called $C_1$, $C_2$ and~$C_3$. One needs to note that although the orbits $O_i$ each have their four points in general position, have (up to renaming the conics):
  \[\begin{array}{l}
   O_2\cup O_3\subset C_1=\defset{\left(x:y:z\right)\in\proj{2}}{x^2+y^2+z^2=0}\\
   O_1\cup O_3\subset C_2=\defset{\left(x:y:z\right)\in\proj{2}}{\zeta_3x^2-(1+\zeta_3)y^2+z^2=0}\\
   O_1\cup O_2\subset C_3=\defset{\left(x:y:z\right)\in\proj{2}}{\zeta_3x^2+y^2-(1+\zeta_3)z^2=0}.\\
  \end{array}\]
  Furthermore, the orbits of size $6$ whose points are not in general position also arise from these conics: the group action preserves the union of lines $L_1\cup L_2\cup L_3$, and these orbits appear as intersections of these three lines with one of these conics. In particular:
  \[\begin{array}{l}
   (0:1:\zeta_{12}^3)\in C_1\cap\left(L_1\cup L_2\cup L_3\right)\\
   (0:1:\zeta_{12})\in C_2\cap\left(L_1\cup L_2\cup L_3\right)\\
   (0:1:\zeta_{12}^5)\in C_3\cap\left(L_1\cup L_2\cup L_3\right).\\
  \end{array}\]
 \end{rmk}
 
 \subsubsection{Building $\bar{G}$-orbits}
 As discussed above, the group $G$ is either equal to $T$ or is generated by $T$ and an additional element \[\sigma=\left(\begin{array}{ccc}\alpha&0&0\\0&0&\alpha\beta\\0&\alpha\gamma&0\end{array}\right).\]
 
 \begin{prop}
  In the notation above, $\alpha^k=\beta^k=\gamma^k=1$. Furthermore, if $k$ is odd then $G=T$.
 \end{prop}
 \begin{proof}
   By construction, $\sigma^2\in D$, so $\alpha^{2k}=1$. Since $\left(\tau\sigma\right)^2,\left(\tau^2\sigma\right)^2,\left(\sigma\tau\sigma\right)^3\in D$, get $\alpha^k=\beta^k=\gamma^k=1$. Since $\sigma\in G\subset\slgroup{3}{\complex}$, have $\alpha^3\beta\gamma=-1$. This implies that $k$ must be even for such a $\sigma$ to exist.
 \end{proof}

 \begin{lemma}\label{lemma:G-orb}
  Assume $G\neq T$, i.e.\ $G$ is generated by $T$ and the element $\sigma$ above (for some values of $\alpha,\beta,\gamma$). Let $X\in\proj{2}$ be a point, whose $\bar{G}$-orbit has size at most $8$.
  Then one of the following cases occurs:
  \begin{enumerate}
   \item The $\bar{G}$-orbit of $X$ consists of the three points $(1:0:0)$, $(0:1:0)$, $(0:0:1)$.
   \item The $\bar{T}$-orbit and the $\bar{G}$-orbit of $X$ match and consist of $6$ points. There are exactly two $\bar{G}$-orbit of size $6$. The points of one of them lie in general position, and the points of the other one lie on a conic.
   \item The $\bar{T}$-orbit and the $\bar{G}$-orbit of $X$ match and consist of $4$ points in general position. There is exactly one $\bar{G}$-orbit of size $4$.
   \item The $\bar{G}$-orbit of $X$ consists of $8$ points lying on a single conic. There is exactly one $\bar{G}$-orbit of size $8$.
  \end{enumerate}
 \end{lemma}
 \begin{proof}
  The $\bar{G}$-orbit of a point $X\in\proj{2}$ is equal to either a single $\bar{T}$-orbit or a union of exactly two $\bar{T}$-orbits (of the same size). Since the possible $\bar{T}$-orbits have been described in Lemma~\ref{lemma:T-orb}, it remains to see when these can be preserved by the action of $\sigma$. Furthermore, since $T\neq G$, $k$ must be even and the only possible $\bar{T}$-orbits are those in Cases~\ref{T-orb:dist}, \ref{T-orb:6pt} and~\ref{T-orb:4pt} of Lemma~\ref{lemma:T-orb}.
  
  Since the distinguished orbit in Case~\ref{T-orb:dist} is always preserved, one only needs to consider the case $k=2$. Here, must have $\alpha,\beta,\gamma=\pm1$. By multiplying $\sigma$ by elements of $D$, can assume that $\alpha=-1$, $\beta=\gamma=1$. This means that there exists (up to \slgroup{3}{\complex}-conjugation) exactly one such group $G$ containing the proper subgroup $T\cong\altgroup{4}$. It can be seen that this group is equal to the symmetric group \symgroup{4}. 
  
  Referring to the notation in Remark~\ref{rmk:A4:geometry}, the action of $\sigma$ preserves the conic $C_1$ and the orbit $O_1$. The $\bar{T}$-orbits $O_2$ and $O_3$ get combined into a single $\bar{G}$-orbit of size $8$ (contained in $C_1$), and the conics $C_2$ and $C_3$ get interchanged by $\sigma$. The $\bar{T}$-orbit of size $6$ given by $C_1\cap\left(L_1\cup L_2\cup L_3\right)$ is preserved. By a simple computation, one can see that the only other $\bar{T}$-orbit of size $6$ preserved by $\sigma$ is the one containing the point $\left(0:1:1\right)$ --- the orbit's points are in general position.
 \end{proof}
 
 \begin{lemma}\label{lemma:sl3:mono:conclusion}
  Let $\bar{G}\subset\pglgroup{3}{\complex}$ be a finite subgroup whose action on \proj{2} has an orbit \defsetshort{e_1,e_2,e_3} of size $3$, but no orbits of size $1$ or $2$. Assume that $\bar{G}$ has a further orbit of size at most $8$. Then one of the following holds:
  \begin{itemize}
   \item $\bar{G}\cong\cyclgroup{3}\rtimes\cyclgroup{3}$ and it has three additional orbits of size $3$ (with none of the orbits contained in a line). Up to \pglgroup{3}{\complex}-conjugation, such a group is unique.
   \item $\bar{G}\cong \cyclgroup{2}^2\rtimes\cyclgroup{3}\cong\altgroup{4}$. The group has no other orbits of size $3$, three orbits of size $4$ and an infinite family of orbits of size $6$. The points of the size~$4$ orbits lie in general position. There are three orbits of size $6$ with all $6$ points lying on a single conic. The other size $6$ orbits have their points in general position. Up to \pglgroup{3}{\complex}-conjugation, such a group is unique.
   \item $\bar{G}\cong \cyclgroup{2}^2\rtimes\symgroup{3}\cong\symgroup{4}$. The group has no other orbits of size $3$, one orbit of size $4$, two orbits of size $6$ and one orbit of size $8$. The points of the orbit of size $4$ and one of the orbits of size $6$ are in general position. The points of the other two orbits (one of size $6$ and one of size $8$) all lie on a single conic curve. Up to \pglgroup{3}{\complex}-conjugation, such a group is unique.
   \item $\bar{G}\cong\cyclgroup{7}\rtimes\cyclgroup{3}$. The group has  no other orbits of size $3$ and three orbits of size $7$, each of them having its points in general position. Up to \pglgroup{3}{\complex}-conjugation, such a group is unique.
  \end{itemize}

 \end{lemma}
 \begin{proof} Immediate from Lemma~\ref{lemma:T-orb} and Lemma~\ref{lemma:G-orb}. \end{proof}

 \subsection{Primitive group actions}
 
 Now consider finite subgroups of \slgroup{3}{\complex} that act primitively. Since, up to \slgroup{2}{\complex}-conjugation, there are only finitely many of them, the relevant computations can be made using their explicit representations. For this, a computer algebra program (GAP, see~\cite{GAP}) has been used to speed up the computations. Of course, it is possible to repeat these computations by hand.
 
 The classification of primitive subgroups is a very old result, which can be found in~\cite{Blichfeldt17}. However, it is more convenient to use its more modern version: 
 \begin{thm}[see~{\cite[Theorem~A]{Yau-Yu93}}]\label{thm:sl3:classification}
Define the following matrices:
\[\begin{array}{lll}
S=\left(\begin{array}{ccc}1&0&0\\0&\omega&0\\0&0&\omega^{2}\end{array}\right)&
T=\left(\begin{array}{ccc}0&1&0\\0&0&1\\1&0&0\end{array}\right)&
W=\left(\begin{array}{ccc}\omega&0&0\\0&\omega&0\\0&0&\omega\end{array}
\right)\\
&&\\
U=\left(\begin{array}{ccc}\epsilon&0&0\\0&\epsilon&0\\0&0&\epsilon\omega
\end{array}\right)&
Q=\left(\begin{array}{ccc}a&0&0\\0&0&b\\0&c&0\end{array}\right)&
V=\frac{1}{\sqrt{-3}}\left(\begin{array}{ccc}1&1&1\\1&\omega&\omega^{2}
\\1&\omega^{2}&\omega\end{array}\right)\\
\end{array}\]
where $\omega=e^{2\pi i/3}$, $\epsilon^{3}=\omega^{2}$ and
$a,b,c\in\complex$ are chosen arbitrarily, as long as $abc=-1$ and
$Q$ generates a finite group.

Up to conjugation, any finite subgroup of \slgroup{3}{\complex} belongs to one of
the following types:
\begin{enumerate}[(A)]
 \item\label{gptype:A} Diagonal abelian group.
 \item\label{gptype:B} Group isomorphic to an irreducible finite subgroups of
\glgroup{2}{\complex} and not conjugate to a group of type~(\ref{gptype:A}).
 \item\label{gptype:C} Group generated by the group in (\ref{gptype:A}) and $T$
and not conjugate to a group of type~(\ref{gptype:A}) or (\ref{gptype:B}).
 \item\label{gptype:D} Group generated by the group in (\ref{gptype:C}) and $Q$
and not conjugate to a group of types~(\ref{gptype:A})---(\ref{gptype:C}).
 \item\label{gptype:E} Group $E_{108}$ of size $108$ generated by $S$, $T$ and $V$.
 \item\label{gptype:F} Group $F_{216}$ of size $216$ generated by $E_{108}$ and an element $P\defby UVU^{-1}$.
 \item\label{gptype:G} Hessian group $H_{648}$ of size $648$ generated by $F_{216}$ and $U$.
 \item\label{gptype:H} Simple group of size $60$ isomorphic to alternating
group \altgroup{5}.
 \item\label{gptype:I} Simple group of size $168$ isomorphic to permutation
group generated by $\left(1234567\right)$, $\left(142\right)\left(356\right)$,
$\left(12\right)\left(35\right)$.
 \item\label{gptype:J} Group of size $180$ generated by the group in
(\ref{gptype:H}) and $W$.
 \item\label{gptype:K} Group of size $504$ generated by the group in
(\ref{gptype:I}) and $W$.
 \item\label{gptype:L} Group $G$ of size $1080$ with $\bar{G}=G/\defsetspan{W}\cong\altgroup{6}$.
\end{enumerate}
\end{thm}

 \begin{prop}\label{prop:sl3:hessian}
  The projective action of the group $E_{108}$ has two orbits of size $6$ and no other orbits of size at most $8$. The points of each of these orbits are in general position (in the sense of Definition~\ref{defin:genPos}). Projective actions of groups $F_{216}$ and $H_{648}$ to \pglgroup{3}{\complex} have no orbits of size at most $8$.
 \end{prop}
 \begin{proof}
  For clarity, it is worth noting that the three groups discussed in this proposition are projected to \pglgroup{3}{\complex} isomorphically. 
  From the description, it is clear that $G\subset E_{108}\subset F_{216}\subset H_{648}$, where $G$ is a monomial group generated by elements $S$ and $T$ from Theorem~\ref{thm:sl3:classification}. It is clear that any projective orbit of $E_{108}$ must be a union of projective orbits of $\bar{G}$. By Lemma~\ref{lemma:T-orb}, $\bar{G}$ has exactly $4$ orbits of size at most $8$, all of them of size $3$. Calculating the action of the element $V\in E_{108}\setminus G$ on these orbits, one can see that the $\bar{G}$-orbits form two $\bar{E}_{108}$-orbits of size $6$: one containing points 
  $\left(1:0:0\right)$, 
  $\left(0:1:0\right)$, 
  $\left(0:0:1\right)$, 
  $\left(1:1:1\right)$, 
  $\left(1:\zeta_{3}:\zeta_{3}^2\right)$, 
  $\left(1:\zeta_{3}^2:\zeta_{3}\right)$
  and the other containing points
  $\left(1:\zeta_{3}:\zeta_{3}\right)$, 
  $\left(1:\zeta_{3}^2:1\right)$, 
  $\left(1:1:\zeta_{3}^2\right)$, 
  $\left(1:\zeta_{3}^2:\zeta_{3}^2\right)$, 
  $\left(1:1:\zeta_{3}\right)$, 
  $\left(1:\zeta_{3}:1\right)$.
  It is easy to check directly that the points of both of these orbits are in general position.
  
  A similar computation shows that the action of $P\subset F_{216}$ combines these two orbits into an orbit of size $12$. Therefore, $F_{216}$ and $H_{648}$ do not have any orbits of size at most $8$.  
 \end{proof}
 
 \begin{prop}\label{prop:sl3:A5}
  Let $G\cong\altgroup{5}$ be the group described in Theorem~\ref{thm:sl3:classification}, Case~\ref{gptype:H}. Then $\bar{G}$ has one orbit of size $6$ and no other orbits of size at most $8$.
 \end{prop}
 \begin{proof}
  Since $G$ is a simple group, it is clear that $\bar{G}\cong G\cong\altgroup{5}$. This group can be generated by two elements: $\bar{g}_1$ of order $5$ and $\bar{g}_2$ order $3$ (corresponding to elements $(12345),(345)\in\altgroup{5}$).
  
  Consider $\bar{G}_1,\bar{G}_2\subset \bar{G}$ the cyclic subgroups generated by $\bar{g}_1$ and $\bar{g}_2$ respectively. Since both these subgroups are cyclic of prime order, any $\bar{G}_1$-orbit ($\bar{G}_2$-orbit) must have size $1$ or $5$ ($1$ or $3$ respectively). Let $x\in\proj{2}$ be a point, whose $\bar{G}$-orbit has size at most $8$. Then either the $\bar{G}$-orbit has size $5$ or it contains a $\bar{G}_1$-fixed point. Similarly, either the orbit has size $3$ or $6$, or it contains a $\bar{G}_2$-fixed point. Therefore, any $\bar{G}$-orbit of size at most $8$ contains a point $x_0\in\proj{2}$ fixed by either $\bar{G}_1$ or $\bar{G}_2$.
  
  Given $M_1,M_2\in\slgroup{3}{\complex}$ matrices representing $\bar{g}_1$ and $\bar{g}_2$ respectively (via the standard projection), the $\bar{G}_i$-orbits of size $1$ correspond to the eigenvectors of these matrices. Therefore, it suffices to compute the eigenvectors of these matrices and to check the sizes of the corresponding orbits.
  
  The group $\altgroup{5}$ has two irreducible two-dimensional representations, differing by an outer automorphism. For each of them, the matrices $M_i$ each have three $1$-dimensional eigenspaces. One of the six corresponding points on $\proj{2}$ has a $\bar{G}$-orbit of size $6$, one has an orbit of size $10$, while the other four have orbits of sizes $12$ or $20$. The proposition follows. 
 \end{proof}
 
 \begin{prop}\label{prop:sl3:A5:actionProperties}
  Let $\bar{G}=\altgroup{5}$ acting on \proj{2} as discussed in Proposition~\ref{prop:sl3:A5}. Let $p_1,\ldots,p_6$ be the unique $\bar{G}$-orbit of size $6$. Let $L_{ij}$ be the line on \proj{2} through $p_i$ and $p_j$. Then the following hold:
  \begin{itemize}
   \item The points $p_1,\ldots,p_6$ are in general position (in the sense of Definition~\ref{defin:genPos}).
   \item The action of $\bar{G}$ permutes the lines $L_{ij}$ transitively
  \end{itemize}
 \end{prop}
 \begin{proof}
  By direct computation.
 \end{proof}

 \begin{prop}\label{prop:sl3:others}
  Let $G$ be one of the groups described in Theorem~\ref{thm:sl3:classification}, Types~(\ref{gptype:I}) or~(\ref{gptype:L}). Then $\bar{G}$ has no orbits of size at most $8$.
 \end{prop}
 \begin{proof}
  The computation for this result follows the idea described in the proof of Proposition~\ref{prop:sl3:A5}. For groups of Type~(\ref{gptype:I}), look at the generators of orders~$7$ and~$3$. It is easy to check that the corresponding cyclic groups are projected to \pglgroup{3}{\complex} isomorphically, so it suffices to check the orbits corresponding to their eigenvectors. There are $2$ irreducible $3$-dimensional representations of this group, and in both of them these orbits have sizes~$21$, $24$, $28$ or~$56$. Therefore, none of the $\bar{G}$-orbits have sizes at most $8$.
  
  For groups of Type~(\ref{gptype:L}), $G\cong3\altgroup{6}$, and $\bar{G}\cong\altgroup{6}$. The standard generators for $G$ are $A$ of order $2$ and $B$ of order $4$ (see~\cite{ATLAS}). To get generators of coprime orders, instead take generators $B$ and $ABB$ (of order $5$). The rest of the computations follows as above: the relevant $\bar{G}$-orbits have sizes~$36$, $45$, $72$ and~$90$.
 \end{proof}

 \begin{lemma}\label{lemma:sl3:primitive:conclusion}
  Let $G\subset\slgroup{3}{\complex}$ be a finite primitive subgroup, and $\bar{G}\subset\pglgroup{3}{\complex}$ its image under the natural projection. Assume that $\bar{G}$ has an orbit of size at most $8$. Then one of the following holds:
  \begin{itemize}
   \item $\bar{G}$ has exactly two orbits of size $6$, $G\cong E_{108}$ and $\bar{G}\cong E_{108}/\cyclgroup{3}$.
   \item $\bar{G}$ has a single orbit of size $6$ and $\bar{G}\cong\altgroup{5}$.
  \end{itemize}
 \end{lemma}
 \begin{proof}
  Since the action of $G$ is primitive, $G$ must be one of the groups described in Theorem~\ref{thm:sl3:classification}, Types~(\ref{gptype:E}---\ref{gptype:L}). The groups of Types~(\ref{gptype:J}) and~(\ref{gptype:K}) have the same \pglgroup{3}{\complex}-images as groups of Types~(\ref{gptype:H}) and~(\ref{gptype:I}) respectively. The rest of the groups are discussed in Propositions~\ref{prop:sl3:hessian}, \ref{prop:sl3:A5} and~\ref{prop:sl3:others}.
 \end{proof}
 
 \section{Elementary links including \proj{2}}

 Now that the list of groups $\bar{G}$ acting on \proj{2} with an orbit of size at most $8$ has been determined, one needs to see what elementary links arise involving the $\bar{G}$-surface $S=\proj{2}$. Throughout this section, the notation from Figure~\ref{figure:links} will be used to discuss the relevant parts of the elementary links.
 
 \begin{prop}\label{prop:link:orb1}
  Let $G\subset\slgroup{3}{\complex}$ be a finite subgroup, whose action is not irreducible. Then there exists a $\bar{G}$-link $\chi:S=\proj{2}\dashrightarrow S'$ with $S'\neq\proj{2}$.
 \end{prop}
 \begin{proof}
  Since, the action of $G$ is not irreducible, it must preserve a line in $\complex^3$. Therefore, the action of $\bar{G}$ has a fixed point $p\in\proj{2}$. Let \defmapname{\pi}{Z}{S} be the blowup of this point. Then $Z$ has the structure of a $\bar{G}$-conic bundle, and taking $S'=Z$ defines a $\bar{G}$-invariant Type~I elementary link, with $S'\neq\proj{2}$.
 \end{proof}
 
 \begin{prop}\label{prop:link:dist_orb}
  Let $G\subset\slgroup{3}{\complex}$ be a subgroup acting irreducibly monomially. Let $\chi:S=\proj{2}\dashrightarrow S'$ a $\bar{G}$-link with $\pi:Z\rightarrow S$ the blowup of an orbit of size $3$. Then $S'=\proj{2}$.
 \end{prop}
 \begin{proof}
  The surface $Z$ does not have the structure of a $\bar{G}$-conic bundle, so $S'$ must be a blowdown of a number of orbits of $-1$-curves on $Z$. Unless $\pi'$ blows down the exceptional divisors of $\pi$ (making $\chi$ the identity), this is just the standard Cremona involution, making $S'=\proj{2}$.
 \end{proof}
 
 \begin{prop}\label{prop:link:orb3}
  Let $G\subset\slgroup{3}{\complex}$ be a subgroup acting irreducibly monomially. Assume $\bar{G}\cong\cyclgroup{3}\rtimes\cyclgroup{3}$, having four orbits of size $3$. Let $\chi:S=\proj{2}\dashrightarrow S'$ a $\bar{G}$-link. Then $S'=\proj{2}$.
 \end{prop}
 \begin{proof}
  By Lemma~\ref{lemma:sl3:mono:conclusion}, the action of $\bar{G}$ has no other orbits of size at most $8$. Therefore, $Z$ must be a blowup of one or two of the orbits of size $3$ on $S$. If $Z$ is a blowup of a single orbit of size $3$, then $S'=\proj{2}$ by Proposition~\ref{prop:link:dist_orb}. Otherwise, $Z$ is the blowup of $6$ points on $\proj{2}$, and $S'$ is the blowdown of a number of $\bar{G}$-orbits of $-1$-curves on $Z$. There are $9$ such orbits, each containing $3$ curves. Since $S'$ must be $\bar{G}$-minimal, it must be the blowup of two of these orbits. Thus, $S'=\proj{2}$.
 \end{proof}

 \begin{prop}\label{prop:link:orb7}
  Let $G\subset\slgroup{3}{\complex}$ be a subgroup acting irreducibly monomially. Assume $\bar{G}\cong\cyclgroup{7}\rtimes\cyclgroup{3}$, having three orbits of size $7$ and one orbit of size $3$. Let $\chi:S=\proj{2}\dashrightarrow S'$ a $\bar{G}$-link. Then $S'=\proj{2}$.
 \end{prop}
 \begin{proof}
  By the definition of an elementary link, $\pi:Z\rightarrow S$ must be a $\bar{G}$-equivariant blowup of at most $8$ points on $S$. Therefore, it is either the blowup of the orbit of size $3$ or the blowup of one of the orbits of size $7$. If it is the blowup of the orbit of size $3$, then $S'=\proj{2}$ by Proposition~\ref{prop:link:dist_orb}. Assume $Z$ is the blowup of one of the size $7$ orbits. This makes $\chi$ the standard Bertini involution (see, for example,~\cite[Section~2.4]{Dolgachev-Iskovskikh09}): $Z$ is the Del~Pezzo surface of degree $2$, which is the double cover of \proj{2} ramified in a quartic. The involution of the double cover interchanges orbits of $-1$-curves (of size~$7$), giving rise to $\chi$. In this case, $S'=\proj{2}$.
 \end{proof}

 \begin{prop}\label{prop:link:orb4}
  Let $G=\altgroup{4}\mbox{ or }\symgroup{4}\subset\slgroup{3}{\complex}$, acting irreducibly. Then there exists a $\bar{G}$-link $\chi:S=\proj{2}\dashrightarrow S'$ with $S'\neq\proj{2}$.
 \end{prop}
 \begin{proof}
  By Lemma~\ref{lemma:sl3:mono:conclusion}, the action of $\bar{G}$ on $S$ has an orbit of size $4$, whose points lie in general position. Let $p_1,\ldots p_4\in\proj{2}$ be such an orbit, and let \deffunname{\pi}{Z}{S} be the blowup of these points. Then $Z$ has the structure of a $\bar{G}$-conic bundle. Therefore, taking $S'=Z$ defines a Type~I elementary $\bar{G}$-link, with $S'\neq\proj{2}$.
 \end{proof}

 \begin{prop}\label{prop:link:E108}
  Let $G=E_{108}\subset\slgroup{3}{\complex}$, acting primitively. Let $\chi:S=\proj{2}\dashrightarrow S'$ a $\bar{G}$-link. Then $S'=\proj{2}$.
 \end{prop}
 \begin{proof}
  By Lemma~\ref{lemma:sl3:primitive:conclusion}, the action of $\bar{G}$ on $S$ has two orbits of size $6$ and no other orbits of size at most $8$. Therefore, $Z$ is the blowup of one of the size $6$ orbits. It is possible to see that here $Z$ is the Fermat cubic surface. Call the orbit's points $p_1,\ldots,p_6\in S$. The map \deffunname{\pi'}{Z}{S'} must be the blowdown of a $\bar{G}$-orbit of $-1$-curves on $Z$. One can easily see that it must be the blowdown of the curves $\tilde{C}_i$, the strict transforms (under $\pi$) of the conics $C_i\subset S$, each passing through exactly $5$ of the $p_i$.
  Therefore, $S'=\proj{2}$.
 \end{proof}

 \begin{prop}[for a different proof, see~\cite{CheltsovShramov16}]\label{prop:link:A5}
  Let $G$ be the group isomorphic to $\altgroup{5}$ acting on $\complex^3$ primitively. Let $\chi:S=\proj{2}\dashrightarrow S'$ be a non-trivial $\bar{G}$-link. Then $\chi$ is an elementary link of Type~II and $S'=\proj{2}$.
 \end{prop}
 \begin{proof}
  The group \altgroup{5} has exactly two irreducible $3$-dimensional complex representations, both of them producing the same subgroup of \slgroup{3}{\complex} (up to an outer automorphism). The action of this subgroup is primitive. By Proposition~\ref{prop:sl3:A5}, the action of $\bar{G}$ has exactly one orbit of size at most $8$. The orbit has size $6$, call its points $p_1,\ldots,p_6$. By Proposition~\ref{prop:sl3:A5:actionProperties}, these points are in general position, so the surface $Z$ in the description of the link must be the blowup of $\proj{2}$ in the six points $p_1,\ldots,p_6$. In fact, it can be seen that $Z$ is the Clebsch cubic surface. The map \deffunname{\pi'}{Z}{S'} must be the blowdown of a $\bar{G}$-orbit of $-1$-curves on $Z$. One can easily see that it must be the blowdown of the curves $\tilde{C}_i$, the strict transforms (under $\pi$) of the conics $C_i\subset S$, each passing through exactly $5$ of the $p_i$. Thus, $S'=\proj{2}$ and $\chi$ is a Type~II link.
 \end{proof}

 \begin{proof}[Proof of Theorem~\ref{thm:main}]\label{lemma:main}
  If the action of $G$ is not irreducible, then $S$ is not $\bar{G}$-birationally rigid by Proposition~\ref{prop:link:orb1}. Thus, assume the action of $G$ is irreducible.
  
  Assume there exists a $\bar{G}$ equivariant link $\chi:S=\proj{2}\dashrightarrow S'$. In the structure of the link, \defmapname{\pi}{Z}{S} is a $\bar{G}$-equivariant blowup of at most $8$ points on $S$. Therefore, the action of $\bar{G}$ on $S$ has an orbit of size at most $8$, and $\bar{G}$ is one of the groups described in Lemmas~\ref{lemma:sl3:mono:conclusion} and~\ref{lemma:sl3:primitive:conclusion}. If $\bar{G}$ is isomorphic to~$\altgroup{4}$ or~$\symgroup{4}$, then $S$ is not $\bar{G}$-birationally rigid by Proposition~\ref{prop:link:orb4}.
  
  Assume that $\bar{G}$ is not isomorphic to \altgroup{4} or \symgroup{4}. Then $S'=\proj{2}$ by Propostions~\ref{prop:link:dist_orb}, \ref{prop:link:orb3}, \ref{prop:link:orb7}, \ref{prop:link:E108} and~\ref{prop:link:A5}. Thus $S$ is $\bar{G}$-birationally rigid.
 \end{proof}

% \bibliography{small_orbits.bib}
% \bibliographystyle{plain}

\end{document}